\documentclass[11pt]{article}

\usepackage{amssymb, amsfonts, amsmath, amsthm}

\newtheorem{theorem}{Theorem}[section]
\newtheorem{lemma}[theorem]{Lemma}
\newtheorem{proposition}[theorem]{Proposition}

\theoremstyle{definition}
\newtheorem{definition}[theorem]{Definition}
\newtheorem{definition-and-remark}[theorem]{Definition and Remark}
\newtheorem{notation}[theorem]{Notation}
\newtheorem{notation-and-remark}[theorem]{Notation and Remark}
\newtheorem{remark}[theorem]{Remark}

\newcommand{\cA}{ {\cal A} }
\newcommand{\bC}{ {\mathbb C} }
\newcommand{\cF}{ {\cal F} }
\newcommand{\bN}{ {\mathbb N} }
\newcommand{\bR}{ {\mathbb R} }
\newcommand{\cP}{{\cal P}}
\newcommand{\cS}{{\cal S}}
\newcommand{\bZ}{ {\mathbb Z} }

\newcommand{\orb}{ \mbox{Orb} }
\newcommand{\term}{ \mbox{term} }

\newcommand{\mnoirr}{ {\underline{m}}^{(irr)} }
\newcommand{\firr}{ f_{\mathrm{irr}} }
\newcommand{\meone}{ {\underline{m}}^{(1)} }
\newcommand{\mek}{ {\underline{m}}^{(k)} }
\newcommand{\unitA}{ 1_{ { }_{\cA} } }

\begin{document}

$\ $

\begin{center}
{\bf\Large  Free probability aspect of irreducible meandric systems,}

\vspace{6pt}

{\bf\Large and some related observations about meanders} 

\vspace{20pt}

{\large Alexandru Nica, \footnote{Research supported 
by a Discovery Grant from NSERC, Canada.} }

{\large University of Waterloo }

\vspace{10pt}

\end{center}

\begin{abstract}

\noindent
We consider the concept of irreducible meandric system introduced by 
Lando and Zvonkin.  We place this concept in the lattice framework of 
$NC(n)$.  As a consequence, we show that the even generating 
function for irreducible meandric systems is the 

\noindent
$R$-transform of $\xi  \eta$, where $\xi$ and $\eta$ are 
classically (commuting) independent random variables, and 
each of $\xi, \eta$ has centred semicircular 
distribution of variance $1$.  Following this point 
of view, we make some observations about the symmetric linear 
functional on $\bC [X]$ which has $R$-transform given by the 
even generating function for meanders.
\end{abstract}

\vspace{6pt}

\begin{center}
{\bf\large 1. Introduction}
\end{center}
\setcounter{section}{1}

A {\em closed meandric system on $2n$ bridges} is a picture obtained by 
independently drawing two non-crossing pairings (a.k.a. 
``arch-diagrams'') of $\{ 1, \ldots , 2n \}$, one of them above 
and the other one below a horizontal line, as exemplified 
in Figure 1.  The combined arches of the two non-crossing pairings 
create a family of disjoint closed curves which wind up and down 
the horizontal line.  If this family consists of precisely 
one curve going through all the points $\{ 1, \ldots , 2n \}$, then 
the meandric system in question is called a {\em closed meander}.  

$\ $

% \\
% \vskip6pt

\begin{center}
  \setlength{\unitlength}{0.3cm}
  \begin{picture}(16,9)
  \thicklines
  \put(-1,4){\line(1,0){16}}
  \put(1,4){\oval(2,2)[t]}
  \put(5,4){\oval(2,2)[b]}
  \put(5,4){\oval(6,6)[b]}
  \put(5,4){\oval(10,10)[b]}
  \put(9,4){\oval(2,2)[t]}
  \put(9,4){\oval(6,6)[t]}
  \put(9,4){\oval(10,10)[t]}
  \put(13,4){\oval(2,2)[b]}
  \end{picture}
\hspace{1cm}
  \begin{picture}(16,9)
  \thicklines
  \put(-1,4){\line(1,0){16}}
  \put(3,4){\oval(2,2)[t]}
  \put(3,4){\oval(6,6)[t]}
  \put(5,4){\oval(2,2)[b]}
  \put(5,4){\oval(6,6)[b]}
  \put(5,4){\oval(10,10)[b]}
  \put(11,4){\oval(2,2)[t]}
  \put(11,4){\oval(6,6)[t]}
  \put(13,4){\oval(2,2)[b]}
  \end{picture}

$\ $

$\ $

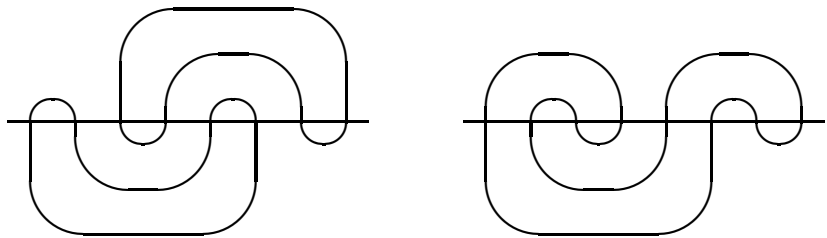
{\bf Figure 1.}  {\em Two closed meandric systems on $8$ bridges,}

{\em where one of them (on the right) is a closed meander.}
\end{center}

$\ $

Let $\meone_n$ denote the number of closed meanders
on $2n$ bridges.  Determining the asymptotic behaviour of the 
sequence $( \meone_n )_{n=1}^{\infty}$ is known to be a difficult 
problem -- see e.g. \cite{DGG1997}, or Section 3.4 of the monograph 
\cite{LZ2004}.  In particular, the constant
\begin{equation}   \label{eqn:1a}
c^{ (1) } := \limsup_{n \to \infty} 
\Bigl( \, \meone_n \, \Bigr)^{1/n} 
\end{equation}  
(reciprocal of radius of convergence for the generating
function of the $\meone_n$) is not known precisely.  
Numerical experimentation gives $c^{ (1) } \approx 12.26$.

In the paper \cite{LZ1992}, Lando and Zvonkin considered the concept 
\footnote{ Henceforth we will implicitly assume the adjective 
``closed'',   and we will just write ``meandric system''
and ``meander'' to mean ``closed meandric system'' and 
``closed meander'', respectively.}
of {\em irreducible meandric system} on $2n$ bridges.  Every meander is 
in particular an irreducible meandric system; hence the number 
$\mnoirr_n$ of irreducible meandric systems on $2n$ bridges is 
an upper bound for $\meone_n$, and the constant
\begin{equation}  
c^{ (irr) } := \limsup_{n \to \infty} 
\Bigl( \, \mnoirr_n \, \Bigr)^{1/n} 
\end{equation} 
is an upper bound for $c^{ (1) }$ of (\ref{eqn:1a}).  Interestingly 
enough, Lando and Zvonkin could determine $c^{ (irr) }$ precisely, 
namely
\begin{equation}   \label{eqn:1b}
c^{ (irr) } = \bigl( \pi / ( 4- \pi ) \bigr)^2 \approx 13.39
\end{equation} 
The equality (\ref{eqn:1b}) was obtained by finding a functional equation 
satisfied by the power series
\begin{equation}   \label{eqn:1c}
1 + \sum_{n=1}^{\infty} \, \mnoirr_n \, z^n,
\end{equation} 
which could then be used to determine the radius of convergence of the series.

In the present paper we place the concept of irreducible meandric 
system in the framework of lattice operations on $NC(n)$, the 
lattice of non-crossing partitions of $\{ 1, \ldots , n \}$.  This 
is done via a natural bijective correspondence (``the doubling 
construction'') between $NC(n)$ and the set of non-crossing 
pairings of $\{ 1, \ldots , 2n \}$, and leads to the following:

$\ $

\begin{theorem}    \label{thm:11}
For every $n \in \bN$, the number $\mnoirr_n$ of irreducible meandric 
systems on $2n$ bridges can be described as
\begin{equation}   \label{eqn:1d}
\mnoirr_n  \, = \ \vline \, \{ ( \pi , \rho ) \in NC(n)^{2} \mid 
\pi \vee \rho = 1_n \mbox{ and } \pi \wedge \rho = 0_n \} \ \vline \, ,
\end{equation} 
where ``$\vee$'' and ''$\wedge$'' are the join and respectively 
meet operations on $NC(n)$, while $0_n, 1_n$ are the minimal and 
respectively maximal element of $NC(n)$.
\end{theorem}

$\ $

In connection to the above, it turns out that a close relative of 
the power series from (\ref{eqn:1c}) has a neat free probabilistic 
interpretation, as an $R$-transform (the counterpart in free 
probability for the concept of characteristic function of a 
random variable).  More precisely, denoting
\begin{equation}   \label{eqn:1e}
\firr (z) := \sum_{n=1}^{\infty} \mnoirr_n \, z^{2n},
\end{equation} 
one has the following:

$\ $

\begin{theorem}   \label{thm:12}
The series $\firr$ from (\ref{eqn:1e}) is the $R$-transform of the 
product $\xi  \eta$, where $\xi$ and $\eta$ are classically 
(commuting) independent random variables, and each of $\xi$ and 
$\eta$ has centred semicircular distribution of variance $1$.
\end{theorem}

$\ $

Theorem \ref{thm:12} can be obtained as a rather straightforward
application of a result of Biane and Dehornoy \cite{BD2014}.

We note that, in view of Theorem \ref{thm:12}, the functional 
equation found by Lando and Zvonkin (when written for the series 
$\firr$) becomes precisely the functional equation which is known 
to always be satisfied by the $R$-transform of a real random 
variable -- see e.g. the discussion on pages 269-270 of the 
monograph \cite{NS2006}.  Moreover, the calculation of radius 
of convergence made in \cite{LZ1992} suggests a method for 
determining, more generally, the radius of convergence for 
$R$-transforms of certain random variables with ``nice'' 
moment-generating functions.

Returning to the analogy between the sequences
$( \meone_n )_{n=1}^{\infty}$ and 
$( \mnoirr_n )_{n=1}^{\infty}$, it is then natural to
consider the power series $f_1$ which is analogous to 
$\firr$ from Equation (\ref{eqn:1e}), but has the meander number 
$\meone_n$ (instead of $\mnoirr_n$) as coefficient of $z^{2n}$.
Theorem \ref{thm:12} suggests that we write $f_1$ as an 
$R$-transform.  We can in any case do that on an algebraic level
-- that is, we can get $f_1$ as $R$-transform of a linear functional 
$\nu : \bC [X] \to \bC$ which is {\em defined} via the 
requirement that $R_{\nu} = f_1$.  The final section of the paper
is devoted to making some observations about this functional 
$\nu$: on the one hand we identify some sets of ``strictly 
non-crossing'' meandric systems which are counted by the even
moments of $\nu$, and on the other hand we observe that 
\[
\nu = \lim_{t \to 0} \nu_t^{\boxplus 1/t}
\ \ \mbox{ (limit in moments)}
\]
where $\boxplus$ refers to the operation of free additive 
convolution and $( \nu_t )_{t \in (0, \infty )}$ (defined 
precisely in Notation \ref{def:55} of the paper) is a family
of linear functionals of independent interest.

$\ $

Besides the present introduction, the paper has four other sections.
After a brief review of $NC(n)$ in Section 2, the proof of 
Theorem \ref{thm:11} is given in Section 3, then the proof and some 
comments around Theorem \ref{thm:12} are given in Section 4.  The 
final Section 5 presents the related observations about meanders
that were mentioned in the preceding paragraph.

$\ $

$\ $

\begin{center}
{\bf\large 2. Background on non-crossing partitions}
\end{center}
\setcounter{section}{2}
\setcounter{theorem}{0}
\setcounter{equation}{0}

In this section we do a brief review, mostly intended for setting the 
notations, of a few basic facts about the lattices of non-crossing 
partitions $NC(n)$.  For a more detailed discussion of this topic, we 
refer the reader to Lectures 9 and 10 of the monograph \cite{NS2006}.

$\ $

\begin{notation}   \label{def:21}
Let $n$ be a positive integer.  

$1^o$ We will work with partitions of the set
$\{ 1, \ldots , n \}$.  Our typical notation for such a 
partition is $\pi = \{ V_1, \ldots , V_k \}$, where 
$V_1, \ldots , V_k$ (the {\em blocks} of $\pi$) are non-empty, 
pairwise disjoint sets with 
$\cup_{i=1}^k V_i = \{ 1, \ldots , n \}$.
Occasionally, we will use the notation ``$V \in \pi$''
to mean that $V$ is one of the blocks of the partition $\pi$.
The number of blocks of $\pi$ is denoted as $| \pi |$.

\vspace{6pt}

$2^o$ We say that a partition $\pi$ of $\{ 1, \ldots ,n \}$ is 
{\em non-crossing} when it is not possible to find 
two distinct blocks $V, W \in \pi$ and numbers $a < b < c < d$ 
in $\{1, \ldots , n \}$ such that $a,c \in V$ and $b,d \in W$.  
This condition amounts precisely to the fact that one can draw
the blocks of $\pi$ without crossings in a picture of the kind
exemplified in Figure 2 below.

\vspace{6pt}

$3^o$ The set of all non-crossing partitions of 
$\{ 1, \ldots , n \}$ is denoted as $NC(n)$.  This is one of the 
many combinatorial structures counted by Catalan numbers 
-- indeed, it is not hard to verify that 
\[
| NC(n) | \,  = C_n := \frac{(2n)!}{n! (n+1)!}
\ \ \mbox{ ($n$-th Catalan number). }
\] 

\vspace{6pt}

$4^o$ On $NC(n)$ we will use the partial order given by 
{\em reverse refinement}: for $\pi, \rho$ we put
\begin{equation}   \label{eqn:21a}
\Bigl( \pi \leq \rho  \Bigr) \ \Leftrightarrow \ 
\left( \begin{array}{c}
\mbox{for every $V \in \pi$ there}  \\
\mbox{exists $W \in \rho$ such that $V \subseteq W$}
\end{array}   \right) .
\end{equation}  
We denote by $0_n$ the partition of $\{ 1, \ldots , n \}$
into $n$ blocks of $1$ element, and we denote by $1_n$ the
partition of $\{ 1, \ldots , n \}$ into $1$ block of $n$ elements. 
These are the minimum and respectively the maximum element in
$\bigl( \, NC(n), \leq \, \bigr)$ (one has 
$0_n \leq \pi \leq 1_n$ for every $\pi \in NC(n)$).
\end{notation}

$\ $

$\ $

$\ $

% \\
% \vskip6pt
\begin{center}
  \setlength{\unitlength}{0.6cm}
  \begin{picture}(6,2)
  \thicklines
  \put(1,1){\line(0,1){2}}
  \put(1,1){\line(1,0){3}}
  \put(2,1){\line(0,1){2}}
  \put(3,2){\line(0,1){1}}
  \put(4,1){\line(0,1){2}}
  \put(5,1){\line(0,1){2}}
  \put(5,1){\line(1,0){1}}
  \put(6,1){\line(0,1){2}}
  \put(0.8,3.2){1}
  \put(1.8,3.2){2}
  \put(2.8,3.2){3}
  \put(3.8,3.2){4}
  \put(4.8,3.2){5}
  \put(5.8,3.2){6}
  \end{picture}

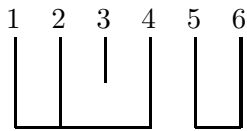
{\bf Figure 2.}  {\em Picture of the partition}

$\pi = \{ \, \{ 1,2,4 \}, \, \{ 3 \}, \, \{ 5,6 \} \, \}
\in NC(6)$.
\end{center}

$\ $

\begin{notation-and-remark}   \label{def:22}
{\em (Lattice properties of $(NC(n), \leq )$).}

\noindent
Let $n$ be a positive integer, and consider the partially ordered set
$(NC(n), \leq )$ from Notation \ref{def:21}.

\vspace{6pt}

$1^o$  The {\em meet} of $\pi, \rho \in NC(n)$ is the partition 
$\pi \wedge \rho$ of $\{ 1, \ldots , n \}$ defined as
\begin{equation}   \label{eqn:22a}
\pi \wedge \rho := \{ V \cap W \mid V \in \pi, W \in \rho, 
V \cap W \neq \emptyset \} .
\end{equation} 
It is easily verified that $\pi \wedge \rho$ belongs to $NC(n)$, 
and is uniquely determined by its properties that:
\[
\left\{  \begin{array}{cl}
\bullet  &  \pi \wedge \rho \leq \pi \mbox{ and } 
            \pi \wedge \rho \leq \rho ;                         \\
         &                                                      \\
\bullet  &  \mbox{If $\lambda \in NC(n)$ is such that 
                  $\lambda \leq \pi$ and $\lambda \leq \rho$, }  \\
         &  \mbox{then it follows that 
                  $\lambda \leq \pi \wedge \rho$. } 
\end{array}  \right.
\]

\vspace{6pt}

$2^{o}$ For every $\pi , \rho \in NC(n)$ there exists a partition 
$\pi \vee \rho \in NC(n)$, called the {\em join} of $\pi$ and $\rho$,
which is uniquely determined by its properties that:
\[
\left\{  \begin{array}{cl}
\bullet  &  \pi \vee \rho \geq \pi \mbox{ and } 
            \pi \vee \rho \geq \rho ;                           \\
         &                                                      \\
\bullet  &  \mbox{If $\lambda \in NC(n)$ is such that 
                   $\lambda \geq \pi$ and $\lambda \geq \rho$, }  \\
         &  \mbox{then it follows that 
                      $\lambda \geq \pi \vee \rho$. } 
\end{array}  \right.
\]
Unlike for $\pi \wedge \rho$, there is no simple explicit formula
describing the blocks of $\pi \vee \rho$.  (It is instructive to
check, for instance, that the join of 
$\{ \, \{ 1,3 \}, \, \{ 2 \}, \, \{ 4 \} \, \}$ and
$\{ \, \{ 1 \}, \, \{ 3 \}, \, \{ 2,4 \} \, \}$
in $NC(4)$ is the partition with one block $1_4$.)
\end{notation-and-remark}  

$\ $

\begin{notation}  \label{def:23}
{\em (Permutation associated to $\pi \in NC(n)$.)}

\noindent
Let $n$ be a positive integer and let
$\cS_n$ denote the group of permutations of $\{ 1, \ldots , n \}$.

\vspace{6pt}

$1^o$ For $\tau \in \cS_n$, we will use the notation $\orb ( \tau )$
for the partition of $\{ 1, \ldots , n \}$ into orbits of $\tau$ (thus
$i$ and $j$ are in the same block of $\orb ( \tau )$ if and only if
there exists $p \in \bN$ such that $\tau^p (i) = j$).  We denote
\[
\# ( \tau )  := \, | \orb ( \tau ) |
\ \ \mbox{ (number of orbits of the permutation $\tau$).}
\]

\vspace{6pt}

$2^{o}$ For $\pi \in NC(n)$ we will denote by $P_{\pi}$ the 
permutation in $\cS_n$ which has $\orb ( P_{\pi} ) = \pi$,
and performs an increasing cycle on every block of $\pi$:
if $V = \{ i_1, i_2, \ldots , i_k \} \in \pi$ with 
$i_1 < i_2 < \cdots < i_k$, then we have
$P_{\pi} (i_1) = i_2, \ldots , P_{\pi} (i_{k-1}) = i_k, \, 
P_{\pi} (i_k) = i_1$. 
\end{notation}

$\ $

\begin{notation-and-remark}   \label{def:24}
{\em (Non-crossing pairings and the doubling construction.)}

\noindent
Let $n$ be a positive integer.  We denote 
\[
NCP(2n) := \{  \sigma \in NC(2n) \mid 
\mbox{every block $W$ of $\sigma$ has $|W| = 2$} \}.
\]
The partitions in $NCP(2n)$ are called {\em non-crossing pairings},
or {\em arch-diagrams} on $2n$ points.  

It is not hard to verify that $| NCP(2n) | = C_n$, the $n$-th Catalan 
number.  Hence $NCP(2n)$ has precisely the same cardinality as $NC(n)$.
One has in fact a natural bijection 
\begin{equation}   \label{eqn:24a}
NC(n) \ni \pi \mapsto A( \pi ) \in NCP (2n), 
\end{equation}  
which goes essentially by ``doubling the points'' in the picture of 
$\pi$, and will therefore be called {\em the doubling construction}
(sometimes also referred to as  ``the fattening construction'').

% \\
% \vskip6pt

\begin{center}
  \setlength{\unitlength}{0.4cm}
  \begin{picture}(25,13)
  \thicklines
  \put(-1,4){\line(1,0){24}}
  \put(3,4){\oval(2,2)[t]}
  \put(7,4){\oval(14,12)[t]}
  \put(9,4){\oval(2,2)[t]}
  \put(9,4){\oval(6,6)[t]}
  \put(19,4){\oval(2,2)[t]}
  \put(19,4){\oval(6,6)[t]}
  \put(0.0,2.8){1}
  \put(2.0,2.8){2}
  \put(4.0,2.8){3}
  \put(6.0,2.8){4}
  \put(8.0,2.8){5}
  \put(10.0,2.8){6}
  \put(12.0,2.8){7}
  \put(14.0,2.8){8}
  \put(16.0,2.8){9}
  \put(17.8,2.8){10}
  \put(19.8,2.8){11}
  \put(21.8,2.8){12}
  \end{picture}

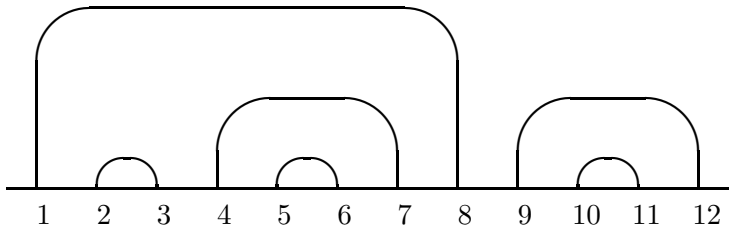
{\bf Figure 3.}  {\em The arch-diagram $A( \pi ) \in NCP(12)$
obtained by performing the}

{\em doubling construction on the 
partition $\pi$ from Figure 2.  (For $1 \leq i \leq 6$, the}

{\em point $i$ in the picture of $\pi$ becomes the interval
$[ 2i-1, 2i ]$ in the picture of $A( \pi )$.)}
\end{center}

$\ $

Formally, the arch-diagram $A( \pi )$ can be introduced by 
indicating how the permutation $P_{A( \pi)} \in \cS_{2n}$ is
described in terms of the permutation $P_{\pi} \in \cS_n$.  
The formula doing this is:
\begin{equation}   \label{eqn:24b}
\left\{ \begin{array}{lcll}
P_{A ( \pi )} (2i)   & = & 2 P_{\pi} (i) -1,   &                  \\
P_{A ( \pi )} (2i-1) & = & 2 P_{\pi}^{-1} (i), & 1 \leq i \leq n.
\end{array}  \right.
\end{equation}
Indeed, it is easy to check that the assignment
\[
2i \mapsto  2 P_{\pi} (i) -1,  \  \ 
2i-1  \mapsto  2 P_{\pi}^{-1} (i), \mbox{ for } 1 \leq i \leq n,
\]
defines a permutation $\tau \in \cS_{2n}$ such that the orbit partition
$\orb ( \tau )$ is in $NCP(2n)$; thus it makes sense to define 
$A( \pi )$ as the unique arch-diagram having $P_{A( \pi )} = \tau$.

From (\ref{eqn:24b}) it is clear that $P_{\pi}$ can be retrieved from
$P_{A( \pi )}$.  This shows that the map $\pi \mapsto A( \pi )$ from
(\ref{eqn:24a}) is one-to-one (hence bijective, since
$| NC(n)| \, = \, | NCP(2n) | $).
\end{notation-and-remark} 

$\ $

$\ $

\begin{center}
{\bf\large 3. Meanders and irreducible meandric systems}
\end{center}
\setcounter{section}{3}
\setcounter{theorem}{0}
\setcounter{equation}{0}

\begin{definition}   \label{def:31}
Let $n$ be a positive integer, and let $\pi, \rho$ be in 
$NC(n)$.  

$1^o$ The {\em meandric system} associated 
to $\pi$ and $\rho$ is the permutation $M_{\pi, \rho} \in \cS_{2n}$ 
defined as follows:
\begin{equation}   \label{eqn:31a}
\left\{ \begin{array}{ccccll}
M_{\pi, \rho} (2i-1) & = &
P_{A ( \pi )} (2i-1) & = & 2 P_{\pi}^{-1} (i), &                   \\
M_{\pi, \rho} (2i)   & = &
P_{A ( \rho )} (2i)  & = & 2 P_{\rho} (i) -1,  & 1 \leq i \leq n. 
\end{array}  \right.
\end{equation}
The number of orbits $\# ( M_{\pi, \rho}$) is called 
{\em number of components} of the meandric system.  

\vspace{6pt}

$2^o$ We will say that $M_{\pi , \rho}$
{\em is a meander} to mean that $\# ( M_{\pi, \rho} ) = 1$. 

\vspace{6pt}

$3^o$  We will say that $M_{\pi , \rho}$ {\em is reducible} to 
mean that there exists a proper subinterval 
$J = \{ a, \ldots , b \} \subset \{1, \ldots , 2n \}$ (with 
$a \leq b$ in $\{ 1, \ldots , 2n \}$ having $b-a < 2n-1$) such 
that $J$ is invariant under the action of $M_{\pi , \rho}$.  
We will say that $M_{\pi , \rho}$ 
{\em is irreducible} to mean that it is not reducible. 
\end{definition}

$\ $

\begin{remark}   \label{rem:32}
$1^{o}$ Let $\pi , \rho$ be as in the preceding definition. 
We record here, for further use, the following immediate 
consequence of the definition of $M_{\pi , \rho}$: for a set
$S \subseteq \{ 1, \ldots , 2n \}$ one has that
\begin{equation}   \label{eqn:32a}
\left(  \begin{array}{c}
\mbox{$S$ is invariant}   \\
\mbox{for $M_{\pi, \rho}$}
\end{array}  \right) \ \Leftrightarrow \  
\left(  \begin{array}{c}
\mbox{$S$ is at the same time}   \\
\mbox{a union of blocks (pairs) of $A( \pi )$}  \\
\mbox{and a union of blocks of $A( \rho )$} 
\end{array}  \right) .
\end{equation} 
Note that (\ref{eqn:32a}) implies, in particular, that every 
set $S \subseteq \{ 1, \ldots , 2n \}$ which is invariant for 
$M_{\pi , \rho}$ must have even cardinality.

\vspace{6pt}

$2^{o}$ Recall from the introduction that for every $n \in \bN$ we have
denoted:
\begin{equation}   \label{eqn:32b}
\meone_n := \ \vline \ \{ ( \pi, \rho ) \in NC(n)^2 \mid
M_{\pi , \rho} \mbox{ is a meander } \} \ \vline \ ,
\end{equation}  
and
\begin{equation}   \label{eqn:32c}
\mnoirr_n := \ \vline \ \{ ( \pi, \rho ) \in NC(n)^2 \mid
M_{\pi , \rho} \mbox{ is irreducible } \} \ \vline \ .
\end{equation}  
It is clear that every meander is in particular an irreducible meandric 
system, but the converse is not true (for instance, the meandric system
depicted on the left side of Figure 1 is irreducible).  Hence 
$\mnoirr_n \geq \meone_n$, where the inequality is generally strict.
The smallest $n$ for which $\mnoirr_n > \meone_n$ is $n=4$ -- the 
reader may find it amusing to verify that there exist precisely 
4 irreducible meandric systems on 8 bridges which are not 
meanders, and this leads to $\mnoirr_4 = 46 = \meone_4 + 4$.
\end{remark} 

$\ $

\begin{lemma}   \label{lemma:3.4}
Let $n$ be a positive integer, let $\pi$ be a partition in $NC(n)$, 
and consider the corresponding arch-diagram $A( \pi ) \in NCP(2n)$.

$1^o$ For $1 \leq p \leq q \leq n$ one has that
\[
\left(  \begin{array}{c}
\mbox{$[ 2p-1 , 2q ] \cap \bZ$ is a}  \\
\mbox{union of blocks of $A( \pi )$}
\end{array}  \right) \ \Leftrightarrow \ 
\left(  \begin{array}{c}
\mbox{$[p , q] \cap \bZ$ is a}  \\
\mbox{union of blocks of $\pi$}
\end{array}  \right) .
\]

\vspace{6pt}

$2^o$ For $1 \leq p < q \leq n$ one has that
\[
\left(  \begin{array}{c}
\mbox{$[ 2p , 2q-1 ] \cap \bZ$ is a}  \\
\mbox{union of blocks of $A( \pi )$}
\end{array}  \right) \ \Leftrightarrow \ 
\left(  \begin{array}{c}
\mbox{$p$ and $q$ belong to}  \\
\mbox{the same block of $\pi$}
\end{array}  \right) .
\]
\end{lemma} 

\begin{proof}  $1^o$ ``$\Rightarrow$''  We must prove that
that if $i \in [p,q] \cap \bZ$, then $P_{\pi} (i)$ still 
belongs to $[p,q]$.  And indeed, for such $i$ we have 
$2i \in [2p-1, 2q] \cap \bZ$, hence our current hypothesis implies 
$P_{A ( \pi )} (i) \in [ 2p-1, 2q ]$.  But then
$P_{\pi} (i) = ( P_{A ( \pi )} (i)  + 1 ) /2 \in [p,q + \frac{1}{2} ]$, 
so (since $P_{\pi} (i)$ is an integer), we conclude that 
$P_{\pi} (i)  \in [p,q] \cap \bZ$, as required.

\vspace{6pt}

$1^o$ ``$\Leftarrow$''  Here we must prove that if
$m \in [2p-1, 2q] \cap \bZ$, then $P_{A( \pi )} (m)$ still 
belongs to $[2p-1, 2q]$.  We distinguish two cases.

\vspace{6pt}

{\em Case 1:  $m$ is even.}
In this case we have $m = 2i$ with $i \in [p,q] \cap \bZ$. 
The current hypothesis entails that 
$P_{\pi} (i)  \in [p,q]$, so we find that 
\[
P_{A( \pi )} (m) = P_{A( \pi )} (2i) =
2 P_{\pi} (i) - 1 \in [ 2p-1, 2q-1 ] \subseteq [ 2p-1, 2q ],
\ \mbox{ as required.}
\]

\vspace{6pt}

{\em Case 2:  $m$ is odd.}
In this case we have $m = 2i-1$ with $i \in [p,q] \cap \bZ$. 
The current hypothesis entails that 
$P_{\pi}^{-1} (i)  \in [p,q]$, so we find that 
\[
P_{A( \pi )} (m) = P_{A( \pi )} (2i-1) =
2 P_{\pi}^{-1} (i) \in [ 2p, 2q ] \subseteq [ 2p-1, 2q ],
\ \mbox{ as required.}
\]

\vspace{6pt}

$2^o$ ``$\Rightarrow$'' We claim there exist $k \geq 1$ and
$p = p_0 < p_1 < \cdots < p_k = q$ such that
\begin{equation}    \label{eqn:34a}
P_{A( \pi )} (2 p_{i-1}) = 2 p_i - 1, \ \ \forall \, 1 \leq i \leq k.
\end{equation}
The points $p_i$ are found recursively, in the way described as follows.
We start with $p_0 = p$ and we look at 
$P_{A( \pi )} (2 p) =: 2 p_1 - 1$.  The current hypothesis gives us 
that $2 p_1 - 1 \in [ 2p, 2q - 1]$, hence that $p < p_1 \leq q$.  
If $p_1 = q$ then we take $k=1$ in (\ref{eqn:34a}) and we are done;
so let us assume that $p_1 < q$.  In this case we remark that 
$[ 2p, 2p_1 - 1] \cap \bZ$ is a union of blocks of $A( \pi )$
(because $A( \pi )$ is non-crossing), hence the set-difference
\[
[ 2p_1, 2q - 1] \cap \bZ =
\Bigl( \, [ 2p, 2q - 1] \cap \bZ \, \Bigr) \setminus
\Bigl( \, [ 2p, 2p_1 - 1] \cap \bZ \, \Bigr) 
\]
must be a union of blocks of $A( \pi )$ as well.  We can thus repeat
the same procedure as above: we look at 
$P_{A( \pi )} (2 p_1) =: 2 p_2 - 1$,  and from the invariance of 
$[ 2p_1, 2q - 1] \cap \bZ$ under $A( \pi )$ we infer that
$p_1 < p_2 \leq q$.  If $p_2 = q$ then we take $k=2$ in (\ref{eqn:34a}) 
and we are done; while if $p_2 < q$, then we look at the invariant set 
$[ 2p_2, 2q - 1] \cap \bZ$  and consider
$P_{A( \pi )} (2 p_2) =: 2 p_3 - 1$, and so on (where, of course, the 
process of finding new points $p_i$ must stop after finitely many steps).

We next compare (\ref{eqn:34a}) against the formula
$P_{A( \pi )} (2 p_{i-1}) = 2 P_{\pi} (p_{i-1}) - 1$ from the definition of
$P_{A( \pi )}$, and we see that the points $ p_0 , p_1 , \ldots , p_k$ 
must satisfy $P_{\pi} ( p_{i-1} ) = p_i$, for all $1 \leq i \leq k$.  This 
implies that all of $p_0 , p_1 , \ldots , p_k$ belong to the same block of 
$\pi$, and (since $p_0 = p$ and $p_k = q$) the required conclusion follows.

\vspace{6pt}

$2^o$ ``$\Leftarrow$''  From the definition of the permutation 
$P_{\pi}$ it follows that there exist $k \geq 1$ and
$p = p_0 < p_1 < \cdots < p_k = q$ such that
$P_{\pi} ( p_{i-1} ) = p_i$, $1 \leq i \leq k$.  We then have
\[
[ 2p, 2q - 1] \cap \bZ 
= \cup_{i=1}^k \Bigl( \, [ 2p_{i-1}, 2p_i - 1 ] \cap \bZ \, \Bigr)
= \cup_{i=1}^k 
\Bigl( \, [ 2p_{i-1}, P_{A( \pi )} (2p_{i-1}) ] \cap \bZ \, \Bigr) .
\]
This in turn implies (by taking into account that $A( \pi )$ is 
non-crossing) that $[ 2p, 2q - 1] \cap \bZ$ is a union of blocks of 
$A( \pi )$, as required.
\end{proof}

$\ $

\begin{proposition}   \label{prop:34}
Let $n$ be a positive integer, let $\pi, \rho$ be in $NC(n)$, 
and let us consider the arch-diagrams 
$A( \pi ), A( \rho ) \in NCP(2n)$ and
the meandric system $M_{\pi , \rho} \in \cS_{2n}$.
The following three statements are equivalent:

\vspace{6pt}

($1$) $M_{\pi , \rho}$ is irreducible.

\vspace{6pt}

($2$) $A( \pi ) \vee A( \rho ) = 1_{2n}$ 
(join considered in $NC(2n)$).

\vspace{6pt}

($3$) $\pi \vee \rho = 1_n$ and $\pi \wedge \rho = 0_n$
(join and meet considered in $NC(n)$).
\end{proposition}

\begin{proof} We will verify the equivalence of the 
complementary statements that:

\vspace{6pt}

($\overline{1}$) $M_{\pi , \rho}$ is reducible;
\hspace{10pt}
($\overline{2}$) $A( \pi ) \vee A( \rho ) \neq 1_{2n}$;
\hspace{10pt}
($\overline{3}$) 
$\pi \vee \rho \neq 1_n$ or $\pi \wedge \rho \neq 0_n$.

\vspace{6pt}

``$( \overline{1} ) \Rightarrow ( \overline{2} )$''.
Let $J$ be a proper subinterval of $\{ 1, \ldots , 2n \}$ which is 
invariant under the action of $M_{\pi , \rho}$.  Thus $J$ is, at the 
same time, a union of blocks of $A( \pi )$ and a union of blocks of 
$A( \rho )$.  Obviously, the same is true for 
$\overline{J} = \{ 1, \ldots , 2n \} \setminus J$, which implies
that the partition $\sigma := \{ J, \overline{J} \} \in NC(2n)$
is such that $\sigma \geq A( \pi )$ and 
$\sigma \geq A( \rho )$.  It follows that 
$A( \pi ) \vee A( \rho ) \leq \sigma$ and hence that
$A( \pi ) \vee A( \rho ) \neq 1_{2n}$, as required.

\vspace{6pt}

``$( \overline{2} ) \Rightarrow ( \overline{3} )$''.
Let us denote $A( \pi ) \vee A( \rho ) =: \sigma \in NC(2n)$.  
Every non-crossing partition has interval blocks, hence we can 
find $1 \leq a \leq b \leq 2n$ such that 
$J := [a,b] \cap \bZ$ is a block of $\sigma$.  Observe that 
$J \neq \{ 1, \ldots , 2n \}$ (since $\sigma \neq 1_{2n}$).
Thus $J$ is a proper subinterval of $\{ 1, \ldots , 2n \}$ which
is, at the same time, a union of blocks of $A( \pi )$ and a union 
of blocks of $A( \rho )$.  We distinguish two possible cases.
\vspace{6pt}

{\em Case 1:  $\min (J)$ is an odd number.}
In this case, $J$ must be of the form
$J := [2p-1, 2q ] \cap \bZ$ for some $1 \leq p \leq q \leq n$.
Lemma \ref{lemma:3.4}.1 gives us that $V := [p, q] \cap \bZ$ 
is at the same time a union of blocks of $\pi$ and a union of 
blocks of $\rho$.  Note that $V \neq \{ 1, \ldots , n \}$, since
$J \neq \{ 1, \ldots , 2n \}$.  Then 
$\lambda := \{ V, \{ 1, \ldots , n \} \setminus V \}$ is in 
$NC(n)$, has $| \lambda | = 2$, and is such that  
$\pi \leq \lambda$ and $\rho \leq \lambda$; hence 
$\pi \vee \rho \leq  \lambda$, implying $\pi \vee \rho \neq  1_n$,
and ($\overline{3}$) holds.

\vspace{6pt}

{\em Case 2:  $\min (J)$ is an even number.}
In this case, $J$ must be of the form
$J := [2p, 2q - 1 ] \cap \bZ$ for some $1 \leq p < q \leq n$.
Lemma \ref{lemma:3.4}.2 gives us that $p$ and $q$ belong to the 
same block of $\pi$, and also that they belong to the same block of
$\rho$.  This implies $\pi \wedge \rho \neq 0_n$ (as $p,q$ are in 
the same block of $\pi \wedge \rho$), and ($\overline{3}$) holds
in this case as well.

\vspace{10pt}

``$( \overline{3} ) \Rightarrow ( \overline{1} )$''.
Here we must verify that either of the hypotheses 
$\pi \vee \rho \neq 1_n$ or $\pi \wedge \rho \neq 0_n$
imply the reducibility of $M_{\pi, \rho}$.

\vspace{6pt}

{\em Claim 1.}  If $\pi \vee \rho \neq 1_n$, then 
$M_{\pi, \rho}$ is reducible. 

{\em Verification of Claim 1.}  Let us denote 
$\pi \vee \rho =: \lambda$.  Every non-crossing partition has 
interval blocks, hence we can find $1 \leq p \leq q \leq n$ such that 
$[ p, q ] \cap \bZ$ is a block of $\lambda$.  Since $\pi \leq \lambda$,
it follows that $[ p, q ] \cap \bZ$ is a union of blocks of $\pi$, 
and Lemma \ref{lemma:3.4}.1 then gives us that 
$J := [2p-1, 2q ] \cap \bZ$ is a union of blocks of $A( \pi )$. 
In the same way we obtain that $J$ is a union of blocks of 
$A( \rho )$.  Note that $J \neq \{ 1, \ldots , 2n \}$ (from
$J = \{ 1, \ldots , 2n \}$ we would infer $p=1, q=n$, hence that
$\lambda = 1_n$).  Thus $J$ is a proper subinterval of 
$\{ 1,\ldots ,2n \}$ which is invariant under $M_{\pi, \rho}$, and 
Claim 1 follows.

\vspace{6pt}

{\em Claim 2.}  If $\pi \wedge \rho \neq 0_n$, then 
$M_{\pi, \rho}$ is reducible. 

{\em Verification of Claim 2.}
$\pi \wedge \rho$ has blocks that are not singletons, hence we can find
$1 \leq p < q \leq n$ such that $p$ and $q$ are in the same block of
$\pi \wedge \rho$.  These $p$ and $q$ belong to the same block of
$\pi$, hence Lemma \ref{lemma:3.4}.2 gives us that 
$J := [2p, 2q - 1] \cap \bZ$ is a union of blocks of $A( \pi )$.  In 
the same way we find that $J$ is a union of blocks of $A( \rho )$.
Thus $J$ is a proper subinterval of $\{ 1,\ldots ,2n \}$ which is 
invariant under $M_{\pi, \rho}$, and Claim 2 follows.
\end{proof}

$\ $

\begin{remark}   \label{rem:35}

$1^o$ Theorem \ref{thm:11} follows from Proposition \ref{prop:34}, 
by equating the cardinalities of the sets of $( \pi , \rho )$'s 
that are considered in the statements ($1$) and ($3$) of that
proposition.

\vspace{6pt}

$2^o$ Condition ($2$) of Proposition \ref{prop:34} has a nice 
interpretation supporting the idea that irreducible meandric 
systems truly are some kind of counterparts of meanders.  To be 
specific, let $\cP (n)$ denote the set of {\em all} partitions 
of $\{ 1, \ldots , n \}$, crossing or non-crossing, and on 
$\cP (n)$ let us consider the partial order by reverse refinement 
(defined exactly as in formula $(\ref{eqn:21a})$ of 
Notation \ref{def:21}.4).  Then $( \cP (n) , \leq )$ turns out 
to be a lattice, with meet operation ``$\wedge$'' defined exactly 
as for $NC(n)$, by block intersections (same formula as 
$(\ref{eqn:22a})$ of Notation \ref{def:22}).  However, the join 
operation of $\cP (n)$ no longer coincides with the ``$\vee$'' of 
$NC(n)$, and we will denote it (slightly differently) as 
``$\widetilde{\vee}$''.  For instance, the reader may find it 
instructive to note that 
$\{ \, \{ 1,3 \}, \, \{ 2 \}, \, \{ 4 \} \, \}
\, \widetilde{\vee} \,
\{ \, \{ 1 \}, \, \{ 3 \}, \, \{ 2,4 \} \, \}
= \{ \, \{ 1,3 \}, \, \{ 2,4 \} \, \} \in \cP (4)$,
in contrast to the comment about ``$\vee$'' which appeared
in the last sentence of Remark \ref{def:22}.

Once the join $\widetilde{\vee}$ on $\cP (n)$ is put into 
evidence, it is easy to verify that for every $n \in \bN$ 
and $\pi , \rho \in NC(n)$ one has:
\begin{equation}   \label{eqn:35a}
\left(  \begin{array}{c}
M_{\pi , \rho} \mbox{ is a meander}   \\
\mbox{(in the sense of Def. \ref{def:31}.2)}
\end{array}   \right) \ \Leftrightarrow \ 
\Bigl( \, A ( \pi ) \, \widetilde{\vee} \, A( \rho ) = 1_{2n}
\Bigr),
\end{equation}
with $A( \pi ), A( \rho ) \in NCP(2n)$ denoting the 
arch-diagrams associated to $\pi$ and $\rho$, respectively.
By comparing (\ref{eqn:35a}) to condition ($2$) of 
Proposition \ref{prop:34} we see that the concept of irreducible 
meandric system is indeed analogous to the one of meander -- we
only change the lattice where the join operation is 
being considered.
\end{remark} 

$\ $

$\ $

\begin{center}
{\bf\large 4. Counting irreducible meandric systems with 
free cumulants}
\end{center}
\setcounter{section}{4}
\setcounter{theorem}{0}
\setcounter{equation}{0}

The goal of this section is to explain how the power series
$\firr (z) = \sum_{n=1}^{\infty} \mnoirr_n \, z^{2n}$
appears as $R$-transform of a nice probability distribution 
(viewed here as a linear functional on $\bC [X]$).  In order to 
make the presentation self-contained, we first review 
the relevant facts needed about free cumulants and $R$-transforms.

$\ $

\begin{definition-and-remark}   \label{def:41}

Let $\mu : \bC [X] \to \bC$ be linear with $\mu (1) = 1$.  

\vspace{6pt}

$1^{o}$ We will use the notation $( \kappa_n ( \mu ) )_{n=1}^{\infty}$ 
for the sequence of {\em free cumulants} of $\mu$.  This is the 
sequence of complex numbers which is uniquely determined by the 
requirement that 
\begin{equation}    \label{eqn:41a}
\mu (X^n) = \sum_{\pi \in NC(n)} \, \Bigl( 
\, \prod_{V \in \pi} \kappa_{ { }_{|V|} } ( \mu ) \, \Bigr), 
\ \ \forall \, n \in \bN .
\end{equation}
Equation (\ref{eqn:41a}) goes under the name of 
``moment-(free) cumulant'' formula.  For instance for $n \leq 3$ 
it says that
\[
\mu (X) = \kappa_1 ( \mu ), \
\mu (X^2) = \kappa_2 ( \mu ) + \kappa_1 ( \mu )^2, \
\mu (X^3) = \kappa_3 ( \mu ) + 3 \kappa_1 (\mu ) \kappa_2 ( \mu ) 
                             + \kappa_1 ( \mu )^3,
\]
which then yields explicit expressions for the first free cumulants:
\begin{equation}    \label{eqn:41b}
\kappa_1 ( \mu ) = \mu (X), \
\kappa_2 ( \mu ) = \mu (X^2) - \mu (X)^2, \
\kappa_3 ( \mu ) = \mu (X^3) - 3 \mu (X) \mu (X^2) 
                             + 2 \mu (X)^3.
\end{equation}
One can write a formula like in (\ref{eqn:41b}) for $\kappa_n ( \mu )$ 
with general $n \in \bN$, where the occurring coefficients are 
understood in terms of the M\"obius function of $NC(n)$; but we 
will not need this here (the interested reader may check pp. 
175-176 in Lecture 11 of the monograph \cite{NS2006}).

\vspace{6pt}

$2^o$ The power series 
$R_{\mu} (z) := \sum_{n=1}^{\infty} \kappa_n ( \mu ) z^n$
is called the {\em R-transform} of $\mu $.

\vspace{6pt}

$3^o$ The {\em functional equation of the R-transform} says that
\begin{equation}    \label{eqn:41c}
R_{\mu} \bigl( \, z( 1 + M_{\mu} (z)) \, \bigr) = M_{\mu} (z),
\end{equation}
with $R_{\mu}$ as above and 
$M_{\mu} (z) := \sum_{n=1}^{\infty} \mu (X^n) z^n$
(moment-generating series for $\mu$).
For the derivation of (\ref{eqn:41c}) out of the moment-cumulant 
formula (\ref{eqn:41a}), see e.g Theorem 10.23 in \cite{NS2006}.

\vspace{6pt}

$4^{o}$ The functional $\mu$ is said to be {\em symmetric}
when it has $\mu (X^{2n-1}) = 0$, 
$\forall \, n \in \bN$.  An immediate consequence of the 
moment-cumulant formula (\ref{eqn:41a}) is that $\mu$ is symmetric 
if and only if $\kappa_{2n-1} ( \mu ) = 0$,
$\forall \, n \in \bN$. 
\end{definition-and-remark}

$\ $

\begin{definition-and-remark}   \label{def:42}
In this definition, $( \Omega, \cF , P)$ is a probability space
and $\xi , \eta : \Omega \to \bR$ are random variables with
finite moments of all orders.

\vspace{6pt}

$1^o$ Let $\mu : \bC [X] \to \bC$ be the linear functional 
determined by the requirement that 
\[
\mu (X^n) = \int \xi^n \, dP, \ \ \forall \, n \in 
\bN \cup \{ 0 \} .
\]
We will refer to $\mu$ as the {\em distribution} of $\xi$.  The 
free cumulants of $\mu$ (as introduced in the preceding 
definition) are also called {\em free cumulants of $\xi$}, and we 
will use the notation 
\[
\kappa_n ( \xi ) := \kappa_n ( \mu ), \ \ n \in \bN.
\]

\vspace{6pt}

$2^o$ We will say that the random variable $\xi$ is 
{\em centred semicircular of variance $1$} to mean that its 
distribution is $( 2 \pi )^{-1} \sqrt{4 - t^2} \, dt$ on 
$[ -2,2 ]$, i.e that for every $n \in \bN$ one has
\begin{equation}  \label{eqn:42a}
\int \xi^n \, dP = \frac{1}{2 \pi} \int_{-2}^2 
t^n \sqrt{4 - t^2} \, dt = 
\left\{  \begin{array}{ll}
0, & \mbox{ if $n$ is odd,}  \\
C_{n/2} \mbox{ (Catalan number),} & \mbox{ if $n$ is even.} 
\end{array}  \right.
\end{equation}
It is easy to verify (by using the equalities 
$| NCP(2n) | = C_n$, $n \in \bN$) that the free cumulants of
a $\xi$ as in (\ref{eqn:42a}) are 
\begin{equation}  \label{eqn:42b}
\kappa_n ( \xi ) =  
\left\{  \begin{array}{ll}
1, & \mbox{ if $n = 2$,}  \\
0, & \mbox{ otherwise.}
\end{array}  \right.
\end{equation}

$3^o$ Suppose that the random variables $\xi$ and $\eta$ are 
independent, hence that the product $\xi \eta$ has moments
\[
\int ( \xi  \eta )^n \, dP 
= \int \xi^n \, dP  \cdot 
\int \eta^n \, dP, \ \ n \in \bN.
\]
Theorem 1.2 of \cite{BD2014} gives the following formula for 
calculating the free cumulants of $\xi  \eta$ in terms of 
those of $\xi$ and of $\eta$: for every $n \in \bN$ one has
\begin{equation}  \label{eqn:42c}
\kappa_n ( \xi \eta ) = \sum_{
\begin{array}{c}
{\scriptstyle \pi , \rho \in NC(n) \ such}  \\
{\scriptstyle that \ \pi \vee \rho = 1_n}
\end{array} } 
\ \Bigl( \, \prod_{V \in \pi} \kappa_{|V|} ( \xi ) \, \Bigr) \,
\Bigl( \, \prod_{W \in \rho} \kappa_{|W|} ( \eta ) \, \Bigr) .
\end{equation}
\end{definition-and-remark}

$\ $

The next proposition is a rephrasing of Theorem \ref{thm:12} 
from the introduction.

$\ $

\begin{proposition}    \label{prop:43}
Let $\xi , \eta : \Omega \to \bR$ be independent random variables
(as in Remark \ref{def:42}.3), where each of $\xi, \eta$ is 
centred semicircular of variance $1$ (as in Remark \ref{def:42}.2). 
Then the free cumulants of the product $\xi \eta$ are
\begin{equation}   \label{eqn:43a}
\kappa_{2n-1} ( \xi \eta ) = 0 \mbox{ and }
\kappa_{2n} ( \xi \eta ) = \mnoirr_n , \ \ n \in \bN.
\end{equation}
\end{proposition}

\begin{proof}  It is clear that $\xi \eta$ has vanishing
odd moments, which implies that
$\kappa_{2n-1} ( \xi \eta ) = 0$ for all $n \in \bN$.
For an even free cumulant $\kappa_{2n} ( \xi \eta )$ we 
calculate as follows:
\begin{align*}
\kappa_{2n} ( \xi \eta )
& = \sum_{  \begin{array}{c}
{\scriptstyle \sigma , \theta \in NC(2n) \ such}  \\
{\scriptstyle that \ \sigma \vee \theta = 1_{2n} }
\end{array} } 
\ \Bigl( \, \prod_{V \in \sigma} \kappa_{|V|} ( \xi ) \, \Bigr) \,
\Bigl( \, \prod_{W \in \theta} \kappa_{|W|} ( \theta ) \, \Bigr)   \\
& \mbox{ $\ $} \mbox{$\ $ (by Theorem 1.2 of \cite{BD2014} -- 
                 Equation (\ref{eqn:42c}))}                        \\
& = \sum_{  \begin{array}{c}
{\scriptstyle \sigma , \theta \in NCP(2n)  \ such}  \\
{\scriptstyle that \ \sigma \vee \theta = 1_{2n} }
\end{array} } \ 1 
\mbox{$\ $ (by Equation (\ref{eqn:42b}))}                        \\
& = | \, \{ ( \sigma , \theta ) \in NCP(2n)^2 \mid 
              \sigma \vee \theta = 1_{2n} \} \, |                  \\
& = | \, \{ ( \pi , \rho ) \in NC(n)^2 \mid 
           A( \pi ) \vee A( \rho ) = 1_{2n} \} \, |               \\
& \mbox{ $\ $}
\mbox{$\ $ (by writing $\sigma = A( \pi ), \theta = A( \rho ))$}  \\
& = \mnoirr_n  
\mbox{ $\ $ (due to ``$(1) \Leftrightarrow (2)$'' in Proposition 
                                           \ref{prop:34}).}
\end{align*}
\end{proof}

$\ $

\begin{remark}   \label{rem:44}
The paper \cite{BD2014} pays special attention to a sequence
of numbers denoted as $( b_{n,2}^{*} )_{n=1}^{\infty}$, where 
one puts
\begin{equation}   \label{eqn:44a}
b_{n,2}^{*} := 
 \ \vline \, \{  ( \pi , \rho ) \in NC(n)^2 \mid 
\pi \wedge \rho = 0_n \} \ \vline \, , \ \ n \in \bN .
\end{equation}
One of the main points made in \cite{BD2014} about the above
sequence is that it can be neatly identified as a sequence of 
free cumulants:
\begin{equation}   \label{eqn:44b}
b_{n,2}^{*} = \kappa_n ( \, ( \xi  \eta )^2 \, ), \ \ n \in \bN ,
\end{equation}
for the same $\xi, \eta$ as considered in Proposition 
\ref{prop:43}.  

Now, one has a non-trivial result about how the free 
cumulants of the square of a symmetric random variable (here 
$( \xi \eta )^2$) are expressed in terms of the even free 
cumulants of the random variable itself.  This is done via an
equation which resembles the moment-cumulant formula, and 
says in the case at hand that
\begin{equation}    \label{eqn:44c}
\kappa_n ( \, (\xi \eta)^2 \, ) = \sum_{\pi \in NC(n)} \, \Bigl( 
\, \prod_{V \in \pi} \kappa_{ { }_{2|V|} } (\xi \eta) \, \Bigr), 
\ \ \forall \, n \in \bN .
\end{equation}
For the proof of (\ref{eqn:44c}), see Proposition 11.25 
in \cite{NS2006}.  

In view of the interpretations we have for the free 
cumulants on the two sides of Equation (\ref{eqn:44c}), we 
thus arrive to a formula which relates the 
numbers $b_{n,2}^{*}$ to irreducible meandric systems, namely
\begin{equation}   \label{eqn:44d}
b_{n,2}^{*} = \sum_{\pi \in NC(n)} 
\Bigl( \, \prod_{V \in \pi} \mnoirr_{|V|} \, \Bigr),
\ \ \forall \, n \in \bN .
\end{equation}

The numbers $b_{n,2}^{*}$ are part of a larger collection of 
numbers $b_{n,d}^{*}$ with $n,d \in \bN$ -- see Equation (2.6) 
of \cite{BD2014} for a description of $b_{n,d}^{*}$ given in 
terms of non-crossing partitions.  We note that for $d=3$ 
one has 
\begin{equation}   \label{eqn:44e}
b_{n,3}^{*} = 
 \ \vline \, \{  ( \pi_1 , \pi_2, \pi_3 ) \in NC(n)^3 \mid 
\pi_1 \wedge \pi_2 = 0_n \mbox{ and } 
\pi_2 \vee \pi_3 = 1_n \} \ \vline \, , \ \ n \in \bN .
\end{equation}
There is a slight resemblance of Equation (\ref{eqn:44e}) with 
(\ref{eqn:1d}) of Theorem \ref{thm:11}, which prompts the question 
if one could also find a formula relating the numbers 
$b_{n,3}^{*}$ to meandric systems.
\end{remark}

$\ $

\begin{remark}   \label{rem:45}

$1^{o}$ The linear functional involved in Proposition 
\ref{prop:43} (that is, the distribution of $\xi \eta$) has 
vanishing moments of odd order, while its even moment of order 
$2n$ is
\[
\int (\xi  \eta )^{2n} \, dP =
\int \xi^{2n} \, dP \cdot \int \eta^{2n} \, dP 
= C_n^2, \ \ n \in \bN.
\]
Upon writing the functional equation of the $R$-transform 
(Equation (\ref{eqn:41c})) for this particular functional,
one thus gets that
\[
\firr \Bigl( \, z( 1+ \sum_{n=1}^{\infty} C_n^2 z^{2n} ) \, \Bigr)
= \sum_{n=1}^{\infty} C_n^2 z^{2n} .
\]
Modulo some trivial transformations, this is the same functional 
equation as found by Lando and Zvonkin in \cite{LZ1992}.

\vspace{6pt}

$2^{o}$ The method used in \cite{LZ1992} for obtaining the radius 
of convergence of $\firr$ points to a class of functionals 
$\mu : \bC [X] \to \bC$ with tractable radius of convergence for 
$R_{\mu}$, as follows.  Suppose that:

\vspace{6pt}

(i) All the free cumulants $( \kappa_n ( \mu ) )_{n=1}^{\infty}$ 
are real non-negative numbers.

(ii) The moment series $M_{\mu} (z)$ 
has a finite positive radius of convergence $r_o$.

(iii) There exist $c > 0$ and $\beta > 1$ such that (with 
$r_o$ from (ii)) one has $\mu (X^n) \leq c r_o^{-n} n^{- \beta}$, 
for all $n \in \bN$. 

\vspace{6pt}

\noindent
Then it makes sense to consider the finite value
$M_{\mu} (r_o) := \sum_{n=1}^{\infty} \mu (X^n) r_o^n 
\in ( 0, \infty )$, and the radius of convergence of the 
$R$-transform $R_{\mu}$ is equal to $r_1$, where
\begin{equation}   \label{eqn:46a}
r_1 = r_o ( 1+ M_{\mu} (r_o) ) .
\end{equation}
The reason for occurrence of this specific value $r_1$ is that upon
writing the functional equation of the $R$-transfom as 
\[
M_{\mu} (z) = R_{\mu} (w) \ \ \mbox{ for } 
w = z( 1+ M_{\mu} (z)),
\]
and upon letting $z$ and $w$ grow along the positive semiaxes of the 
$z$-plane and $w$-plane, they will hit at the same time the 
singularities that are closest to origin for $M_{\mu}$ and $R_{\mu}$,
respectively.

In the specific case of Proposition \ref{prop:43} (when $\mu$ is
the distribution of $\xi  \eta$), one has 
$M_{\mu} (z) = \sum_{n=1}^{\infty} C_n^2 z^{2n}$ with radius of 
convergence $r_o = 1/4$.  From the asymptotics
$C_n \sim c 4^n n^{-3/2}$ (which folows e.g. from 
Stirling's formula) it follows that in (iii) above we may take 
$\beta = 3$.  The radius of convergence for $R_{\mu} = \firr$ thus 
comes out as
\begin{equation}   \label{eqn:46b}
r_1 = \frac{1}{4} \cdot  \Bigl( \, 1+ \sum_{n=1}^{\infty}
C_n^2 ( 1/4 )^{2n} \, \Bigr) .
\end{equation}
As shown in \cite{LZ1992}, one can determine precisely that 
$1+ \sum_{n=1}^{\infty} C_n^2 ( 1/4 )^{2n} = 4(4- \pi ) / \pi$,
which leads to $r_1 = (4- \pi) / \pi$, and to the value of 
$c^{ (irr) }$ indicated in  Equation (\ref{eqn:1b}).
\end{remark}

$\ $

$\ $

\begin{center}
{\bf\large 5. Counting meanders with free cumulants?}
\end{center}
\setcounter{section}{5}
\setcounter{theorem}{0}
\setcounter{equation}{0}

In this section we look at the framework analogous to the one of 
Theorem \ref{thm:12}, but where instead of the power series 
$\firr$ from Theorem \ref{thm:12} we consider the series 
\begin{equation}   \label{eqn:5a}
f_1 (z) := \sum_{n=1}^{\infty}  \meone_n z^{2n},
\end{equation}
with $\meone_n$ counting the meanders on $2n$ bridges, $n \in \bN$.
More precisely, Theorem \ref{thm:12} says that 
``$\firr = R_{\mu}$'', where $\mu : \bC [X] \to \bC$ is the 
symmetric linear functional with 
$\mu (X^{2n}) = C_n^2$, $n \in \bN$;  so we consider the 
analogous equation ``$f_1 = R_{\nu}$'', which is now used as a
definition, for a functional $\nu$.  The fact that a linear
functional on $\bC [X]$ can be defined by prescribing its 
$R$-transform follows immediately from the moment-cumulant
formula (see e.g. Exercise 16.21 in \cite{NS2006}).

$\ $

\begin{notation}   \label{def:52}
We denote as $\nu : \bC [X] \to \bC$ the linear functional with 
$\nu (1) = 1$ and such that $R_{\nu} = f_1$, the series from 
Equation (\ref{eqn:5a}).  That is, $\nu$ is uniquely determined by
the requirement that its free cumulants are 
\begin{equation}    \label{eqn:52a}
\kappa_{2n-1} ( \nu ) = 0 \ \mbox{ and } \ 
\kappa_{2n} ( \nu ) = \meone_n , \ \  \forall \, n \in \bN .
\end{equation}
\end{notation}  

$\ $

In order to give an alternative description of $\nu$ in terms 
of its moments, we introduce the following concept.

$\ $

\begin{definition}   \label{def:53}
Let $n$ be a positive integer, let $\pi , \rho$ be in $NC(n)$, and consider the 
meandric system $M_{\pi , \rho} \in \cS_{2n}$.  Let 
$\sigma := \orb  \bigl( \, M_{\pi , \rho} \, \bigr)$, the partition of 
$\{ 1, \ldots , 2n \}$ into orbits of $M_{\pi , \rho}$.  If $\sigma$ is 
non-crossing, then we will say that the  meandric system $M_{\pi , \rho}$ is 
{\em strictly non-crossing}.

\vspace{6pt}

\noindent
[For a concrete example, the meandric system depicted on the 
left side of Figure 1 is not strictly non-crossing, since it has
$\orb ( M_{\pi , \rho} ) = 
\{ \, \{ 1,2,5,6 \} , \, \{ 3,4,7,8 \} \, \} \not\in NC(8)$.]
\end{definition}  

$\ $

\begin{proposition}   \label{prop:54}
For every $n \in \bN$, the functional $\nu$ introduced in 
Notation \ref{def:52} has $\nu ( X^{2n-1} ) = 0$ and
\begin{equation}    \label{eqn:54a}
\nu ( X^{2n} ) \, = \ \vline \, \{ ( \pi , \rho ) \in NC(n)^{2} \mid 
M_{\pi , \rho} \mbox{ is strictly non-crossing} \} \ \vline \, .
\end{equation}
\end{proposition}  

\begin{proof}  The vanishing of odd moments of $\nu$ follows from the 
vanishing of its odd free cumulants, as mentioned in 
Remark \ref{def:41}.4.  Here we fix $n \in \bN$ and we address 
the calculation of $\nu ( X^{2n} )$.  We start from the right-hand 
side of Equation (\ref{eqn:54a}), which we write as
\[
\sum_{\sigma \in NC(2n)}^{ } 
\mid \{ ( \pi , \rho ) \in NC(n)^{2} \mid 
\orb ( M_{\pi , \rho} ) = \sigma \} \mid .
\]
Since all orbits of $M_{\pi , \rho}$ have even cardinality,  
the above summation reduces to
\begin{equation}    \label{eqn:54c}
\sum_{\sigma \in NCE(2n)} \
\mid \{ ( \pi , \rho ) \in NC(n)^{2} \mid 
\orb ( M_{\pi , \rho} ) = \sigma \} \mid ,
\end{equation}
where we denoted
\begin{equation}    \label{eqn:nce}
NCE (2n) := \{ \sigma \in NC(2n) \mid 
\ |W| \mbox{ is even, for all $W \in \sigma$} \} .
\end{equation}

Let us momentarily fix a partition 
$\sigma = \{ W_1, \ldots , W_k \} \in NCE (2n)$.  To every 
$( \pi , \rho ) \in NC(n)^{2}$ such that 
$\orb ( M_{\pi , \rho} ) = \sigma$ we can associate a $k$-tuple of
meanders on $|W_1|$,
respectively $|W_2|, \ldots , \, \mbox{respectively } |W_k|$
bridges, in the way described as follows.  For every $1 \leq i \leq k$,
the set $W_i$ is at the same time a union of blocks of $A( \pi )$ and
a union of blocks of $A( \rho )$.  We can thus consider the restrictions 
$A( \pi ) \mid W_i$ and $A( \rho ) \mid W_i$, which become non-crossing
pairings $\sigma_i, \theta_i \in NCP( |W_i| )$ upon the re-numbering 
of the elements of $W_i$ as $1, \ldots , |W_i|$.  We then write 
$\sigma_i = A( \pi_i )$, 
$\theta_i = A( \rho_i )$ with $\pi_i, \rho_i \in NC( |W_i|/2 )$, 
and we note that $M_{\pi_i, \rho_i}$ is a meander (due to the fact 
that $W_i$ is an orbit of $M_{\pi, \rho}$).

The preceding paragraph has put into evidence a natural map 
\[
( \pi , \rho ) \mapsto \Bigl( \, 
( \pi_1 , \rho_1 ), \ldots , ( \pi_k , \rho_k )  \, \Bigr),
\]
going from $\{ ( \pi , \rho ) \in NC(n)^{2} 
                 \mid \orb ( M_{\pi , \rho} ) = \sigma \}$ to 
\begin{equation}    \label{eqn:54e}
\prod_{i=1}^k 
\{ ( \pi_i , \rho_i ) \in NC( |W_i|/2 )^{2} 
\mid  M_{\pi_i , \rho_i} \mbox{ is a meander} \} .
\end{equation}
This map is in fact a bijection.  Indeed, if we start with a 
$k$-tuple 
$\bigl( \, ( \pi_1 , \rho_1 ), \ldots , ( \pi_k , \rho_k )  \, \bigr)$
from the set in (\ref{eqn:54e}), then every 
$( A( \pi_i ), A( \rho_i ) )$ can be re-numbered into a meander on $W_i$, 
and the $k$ meanders thus created will combine together into a meandric 
system with orbit-partition equal to $\sigma$.  (A detail to be emphasized
at this point is that, when putting together the $k$ meanders, we don't
get any crossings.  This holds because $\sigma$ was picked to be in 
$NC(2n)$.  Indeed, from the fact that the blocks of $\sigma$ don't cross 
it follows that there can't be crossings among the re-numbered $A( \pi_i )$'s,
and likewise for the re-numbered $A( \rho_i )$'s.)

The conclusion of the preceding two paragraphs is that, for a fixed 
$\sigma \in NCE(2n)$, we have a bijection between
$\{ ( \pi , \rho ) \in NC(n)^{2} \mid 
\orb ( M_{\pi , \rho} ) = \sigma \}$ and the set from (\ref{eqn:54e}).
Upon equating cardinalities, we infer that
\begin{equation}    \label{eqn:54f}
\mid \{ ( \pi , \rho ) \in NC(n)^{2} \mid 
\orb ( M_{\pi , \rho} ) = \sigma \} \mid \ = \,
\prod_{W \in \sigma} \, \meone_{ {  }_{ |W|/2 } } .
\end{equation}

We now unfix $\sigma$, and plug the equality (\ref{eqn:54f}) into
(\ref{eqn:54c}), to find that the right-hand side of  
(\ref{eqn:54a}) can be written as 
\[
\sum_{\sigma \in NCE(2n)} \, \Bigl( \, 
\prod_{W \in \sigma} \, \meone_{ {  }_{ |W|/2 } } \, \Bigr) .
\]
By taking into account what are the free cumulants of $\nu$, 
we see that the latter expression equals
\[
\sum_{\sigma \in NC(2n)} \, \Bigl( \, 
\prod_{W \in \sigma} \, \kappa_{ {  }_{ |W| } } ( \nu ) \, \Bigr) ,
\]
which gives $\nu (X^{2n})$, as required.
\end{proof}

$\ $

\begin{remark}   \label{rem:54}
The online encyclopedia of integer sequences gives, following 
the paper \cite{J1999}, the meander numbers $\meone_n$ for 
$1 \leq n \leq 24$ (see www.oeis.org, sequence A005315).
Starting from these values, one can use the moment-cumulant 
formula (\ref{eqn:41a}) in order to calculate
\footnote{ I am grateful to Mathieu Guay-Paquet and Franz 
Lehner for their help with computer-aided calculations.}
the even moments of $\nu$ up to order $48$, as listed in 
Table 1 on the next page.  An interesting problem concerning
these moments is to find non-trivial lower bounds for the 
radius of convergence of the series 
$M_{\nu} (z) = \sum_{n=1}^{\infty} \nu( X^{2n} ) z^{2n}$.
This, in turn, could give non-trivial upper bounds for the 
constant $c^{(1)}$ in Equation (\ref{eqn:1a}), via an argument 
like the one mentioned in Remark \ref{rem:45}.2.

The moments listed in Table 1 show that (unfortunately) $\nu$
is not positive definite -- for instance the determinant of
the matrix $[ \, \nu( X^{2i + 2j} )  \, ]_{0 \leq i,j \leq 9}$
is negative.  It would nevertheless be of interest to pursue 
the study of $\nu$ as an analytic object (as a signed measure,
perhaps).

Another observation about $\nu$ is that it relates to a 
family of functionals which are interesting 
in their own right, and are defined as follows.
\end{remark}

$\ $

\begin{notation}   \label{def:55}

$1^{o}$  For every $n \in \bN$ and $k \in \{ 1, \ldots , n \}$ we 
denote
\[
\mek_n \, := \ \vline \, \{ ( \pi , \rho ) \in NC(n)^{2} \mid 
M_{\pi , \rho} \mbox{ has exactly $k$ orbits} \} \ \vline \ \, .
\]
(For $k=1$, this agrees with the notation $\meone_n$ used since the 
introduction.)

\vspace{6pt}

$2^{o}$  Let $t$ be a parameter in $( 0, \infty )$.  We will denote as 
$\nu_t : \bC [X] \to \bC$ the linear functional with $\nu_t (1) = 1$
and which has moments given by
\begin{equation}   \label{eqn:55b}
\nu_t (X^{2n-1}) = 0 \ \mbox{ and } \ 
\nu_t (X^{2n}) = \sum_{k=1}^n \mek_n t^k, \ \ n \in \bN .
\end{equation} 
\end{notation}

$\ $

\begin{remark}   \label{rem:56}
In order to state, in the next proposition, the connection between $\nu$
and the $\nu_t$'s, let us review some more (rather standard) bits of 
terminology.

\vspace{6pt}

{\em (a) $\boxplus$-powers.}
Let $\mu : \bC [X] \to \bC$ be linear with $\mu (1) = 1$, and let $t$ be
in $( 0, \infty )$.  We denote as $\mu^{\boxplus t}$ the linear functional
$\widetilde{\mu} : \bC [X] \to \bC$ which has $\widetilde{\mu} (1) = 1$ and 
is uniquely determined by the requirement that 
$R_{\widetilde{\mu}} (z) = t R_{\mu} (z)$.  The exponential notation 
$\mu^{\boxplus t}$ is meaningful in connection to 
the operation $\boxplus$ of free additive convolution, see for
instance pp. 231-233 in Lecture 14 of \cite{NS2006}.

\vspace{6pt}

{\em (b) Convergence in moments.} 
Let $( \mu_t )_{t \in (0, \infty )}$ and $\mu$ be linear maps from $\bC [X]$ 
to $\bC$, which send $1$ to $1$.   We will write
\begin{equation}   \label{eqn:56a}
\mbox{``$\lim_{t \to 0} \mu_t = \mu$,  in moments''}
\end{equation} 
to mean that $\lim_{t \to 0} \mu_t (X^n) = \mu (X^n)$ 
for all $n \in \bN$.  Upon invoking the 
moment-cumulant formula (\ref{eqn:41a}) it is immediate that, 
equivalently, one can define (\ref{eqn:56a}) via the requirement 
that $\lim_{t \to 0} \kappa_n ( \mu_t ) = \kappa_n ( \mu )$ 
for all $n \in \bN$.
\end{remark} 

$\ $

$\ $

\begin{center}
\begin{tabular}{|c|l|l|l|}
\hline
     &                &           &                         \\
$n$  &  Free cumulant &  Moment   &  Ratio                  \\
     &  $\kappa_{2n} ( \nu ) = \meone_n$  & 
        $\nu( X^{2n} ) $                  &
        $\nu ( X^{2n} ) / C_n^2$                            \\
     &                &           &           \\   \hline
     &                &           &           \\   

1  & 1    & 1            & 1.00000           \\
2  & 2    & 4            & 1.00000           \\
3  & 8    & 25           & 1.00000           \\
4  & 42   & 192          & 0.97959           \\
5  & 262  & 1664         & 0.94331           \\
6  & 1828 & 15626        & 0.89681           \\
7  & 13820        & 155439            & 0.84459     \\
8  & 110954       & 1615208           & 0.78987     \\
9  & 933458       & 17371372          & 0.73486     \\
10 & 8152860      & 192116692         & 0.68101     \\
11 & 73424650     & 2174556080        & 0.62925     \\
12 & 678390116    & 25101780538       & 0.58013     \\
13 & 6405031050   & 294692569630      & 0.53396     \\
14 & 61606881612  & 3510877767198     & 0.49085     \\
15 & 602188541928 & 42371895120585    & 0.45081     \\

16 & 5969806669034         & 517281396522616          & 0.41377  \\
17 & 59923200729046        & 6380271752428956         & 0.37960  \\
18 & 608188709574124       & 79428025047086276        & 0.34816  \\
19 & 6234277838531806      & 997137221492794404       & 0.31926  \\
20 & 64477712119584604     & 12614196796924143524     & 0.29276  \\
21 & 672265814872772972    & 160696941192856063186    & 0.26845  \\
22 & 7060941974458061392   & 2060412248079723985072   & 0.24619  \\
23 & 74661728661167809752  & 26575640310738797507800  & 0.22581  \\
24 & 794337831754564188184 & 344671815256362419882958 & 0.20715  \\
   &                       &                    &      \\ \hline

\end{tabular}

$\ $

$\ $

{\bf Table 1.}  
{\em Even free cumulants and even moments
of the linear functional $\nu$, up

to order 48.  The rightmost column of the table shows the 
probability that a 

random
meandric system on $2n$ bridges is strictly non-crossing,
for $1 \leq n \leq 24$.}
\end{center}

\newpage

\begin{proposition}   \label{prop:57}
One has
$\lim_{t \to 0} \nu_t^{\boxplus \, 1/t} = \nu , \mbox{ in moments,}$
where $\nu_t$ and $\nu$ are as in Notations \ref{def:55} and 
\ref{def:52}, respectively.
\end{proposition} 

\begin{proof}  We will prove the convergence of free cumulants,
\begin{equation}   \label{eqn:57a}
\lim_{t \to 0} \kappa_n ( \nu_t^{\boxplus 1/t} ) = \kappa_n ( \nu ), 
\ \ \forall \, n \in \bN .
\end{equation} 
For $n$ oddd, (\ref{eqn:57a}) holds trivially, because $\nu$ and the 
$\nu_t$'s are symmetric functionals.  For $n$ even, $n= 2p$, the 
limit in (\ref{eqn:57a}) amounts to 
\begin{equation}   \label{eqn:57b}
\lim_{t \to 0}  \frac{1}{t} \kappa_{2p} ( \nu_t ) = \meone_p .
\end{equation} 
We will obtain this as a consequence of the following stronger claim.

\vspace{6pt}

{\em Claim.} For every $p \in \bN$, there exists a polynomial 
$Q_p \in \bZ [t]$, with $Q_p (0) = 0$ and $Q_p ' (0) = \meone_p$,
such that $\kappa_{2p} (\nu_t ) = Q_p (t)$ for all $t \in ( 0, \infty )$.

{\em Verification of Claim.} By induction on $p$.  For $p=1$ we have
$\kappa_2 ( \nu_t ) = \nu_t (X^2) - \nu_t (X)^2 = t$, hence we can
take $Q_1 (t) = t = \meone_1 t$.

Induction step: we fix $p \geq 2$ and we verify the claim for this $p$,
by assumming it was already verified for $1, \ldots , p-1$.  For every
$t \in ( 0, \infty )$, the moment-cumulant formula says that 
\[
\nu_t (X^{2p}) =  \sum_{\sigma \in NC(2p)} \, \Bigl( 
\, \prod_{W \in \sigma} \kappa_{ { }_{|W|} } ( \nu_t ) \, \Bigr).
\]
Since $\nu_t$ is symmetric, the latter sum has in fact only contributions
from partitions in $NCE(2p)$ (same notation as in Equation (\ref{eqn:nce})
from the proof of Proposition \ref{prop:54}).  By separating the term which 
corresponds to $\sigma = 1_{2p}$, we find that
\begin{equation}   \label{eqn:57c}
\kappa_{2p} ( \nu_t ) = \nu_t (X^{2p}) -  
\sum_{ \begin{array}{c}
{ \scriptstyle \sigma \in NCE(2p) }  \\
{ \scriptstyle \sigma \neq 1_{2p} } 
\end{array} } \ \Bigl( 
\, \prod_{W \in \sigma} \kappa_{ { }_{|W|} } ( \nu_t ) \, \Bigr).
\end{equation}
The induction hypothesis allows us to replace the sum which is subtracted in
(\ref{eqn:57c}) with
\[
\sum_{\begin{array}{c}
{ \scriptstyle \sigma \in NCE(2p) }  \\
{ \scriptstyle \sigma \neq 1_{2p} } 
\end{array} } \ \Bigl( 
\, \prod_{W \in \sigma} Q_{ { }_{|W|/2} } ( t ) \, \Bigr) =: U(t),
\]
where $U \in \bZ [t]$ has $U(0) = U' (0) = 0$.  If on the right-hand side of
(\ref{eqn:57c}) we also substitute 
$\nu_t (X^{2p}) = \meone_p t + \sum_{k=2}^p \mek_p t^k$, it 
clearly follows that $\kappa_{2p} ( \nu_t )$ has indeed the form required
by the claim.
\end{proof}

$\ $

\begin{remark}   \label{rem:58}
It is natural to ask: for what values of $t$ is $\nu_t$ positive 
definite?  Proposition \ref{prop:57} shows this cannot hold for
$t \to 0$ (if there would exist a sequence $t_n \to 0$ with 
$\nu_{t_n}$ positive definite, then it would follow that $\nu$ is 
positive definite as well).  On the other hand, there are values 
of $t \geq 1$ for which $\nu_t$ is sure to be positive definite 
because it admits an {\em operator model} (that is, it arises as 
scalar spectral measure for a selfadjoint operator on Hilbert space). 
The largest known set of such $t$'s appears to be
$\{ 2 \cos \frac{\pi}{n} \mid n \geq 3 \} \cup [ 2, \infty )$;
for $t$ in this set, an operator model for $\nu_t$ is described in 
\cite{CJS2014} (see discussion preceding Proposition 3.1 of that 
paper).  Some other operator models (or random matrix models) for 
$\nu_t$ are known in the special case when $t \in \bN$: see 
Section 4 of \cite{FS2013}, or Section 4 of the physics paper 
\cite{M1995}; the latter model is also described in Section 6.2
of \cite{D2001}.  

The next proposition presents a version of the model from 
\cite{M1995}, \cite{D2001}, which is placed in  the framework of 
a $*$-probability space (that is, $\cA$ is a unital $*$-algebra
over $\bC$, and $\varphi : \cA \to \bC$ is linear with 
$\varphi ( \unitA ) = 1$ and $\varphi ( a^{*}a ) \geq 0$ for all 
$a \in \cA$).  The interesting point of the proposition is that 
it involves tensors products of elements from a free family -- 
a mixture of classical and free probability which may provide a 
good setting for further study of meandric systems.
\end{remark}

$\ $

\begin{proposition}   \label{prop:59}
Let $d$ be a positive integer.  Suppose that $a_1, \ldots , a_d$ is a 
free family of selfadjoint elements in a $*$-probability space 
$( \cA , \varphi )$, such that every $a_i$ ($1 \leq i \leq d$) has 
centred semicircular distribution of variance $1$.  Consider the 
$*$-probability space 
$( \cA \otimes \cA , \varphi \otimes \varphi)$, and the selfadjoint
element
\begin{equation}   \label{eqn:59a}
x = a_1 \otimes a_1  + a_2 \otimes a_2 + \cdots + a_d \otimes a_d 
\in \cA \otimes \cA .
\end{equation}
Then $x$ has distribution $\nu_d$ with 
respect to $\varphi \otimes \varphi$.
\end{proposition}

\begin{proof}  The conclusion of the proposition amounts to the 
fact that for every $n \in \bN$ one has
\begin{equation}  \label{eqn:59b}
( \varphi \otimes \varphi ) (x^{2n-1}) = 0
\ \mbox{ and } \
( \varphi \otimes \varphi ) (x^{2n}) = 
\sum_{k=1}^n \mek_n d^k.
\end{equation}
Throughout the proof we fix an $n \in \bN$ for which we will 
verify the second formula (\ref{eqn:59b}) (the easy verification 
that $( \varphi \otimes \varphi ) (x^{2n-1}) = 0$ is
left to the reader).

We start by expanding $( a_1 \otimes a_1  + a_2 \otimes a_2 
+ \cdots + a_d \otimes a_d )^{2n}$, and by applying 
$\varphi \otimes \varphi$ to the result, to get
\begin{equation}  \label{eqn:59c}
( \varphi \otimes \varphi ) (x^{2n}) = 
\sum_{i(1), \ldots , i(2n) =1}^d
\Bigl( \, \varphi ( a_{i(1)} \cdots a_{i(2n)} ) \, \Bigr)^2.
\end{equation}

Let us momentarily fix a $(2n)$-tuple 
$(i(1), \ldots , i(2n)) \in \{ 1, \ldots , d \}^{2n}$.  The 
moment-cumulant formula for several variables (for which we refer to
Lecture 11 of \cite{NS2006}) expresses the moment 
$\varphi ( a_{i(1)} \cdots a_{i(2n)} )$ as a certain summation over 
$NC(2n)$,
\begin{equation}  \label{eqn:59d}
\varphi ( a_{i(1)} \cdots a_{i(2n)} ) = \sum_{\sigma \in NC(2n)} \
\term_{\sigma}.
\end{equation}  
Due to the free independence of $a_1, \ldots , a_d$
and to the special form of the free cumulants of the $a_i$'s (namely
$\kappa_2 (a_i) = 1$ and $\kappa_p (a_i) = 0$ for $p \neq 2$), it
turns out that in (\ref{eqn:59d}) we always have 
$\term_{\sigma} \in \{ 0, 1 \}$, with
\[ 
( \term_{\sigma} = 1 ) \ \Leftrightarrow \ \left(
\begin{array}{c} 
\mbox{$\sigma \in NCP (2n)$, and for every }  \\
\mbox{$W = \{ p,q \} \in \sigma$ one has $i(p) = i(q)$ }
\end{array}   \right)  .
\]
It comes in handy to introduce here a notation, 
say ``$\sigma \leq \ker i$'' to mean 
\footnote{ It is customary to denote by $\ker i$ the 
partition of $\{ 1, \ldots , 2n \}$ defined via the 
requirement that for $1 \leq p,q \leq 2n$ one has:
``($p,q$ belong to the same block of $\ker i$) 
$\Leftrightarrow$ $i(p) = i(q)$''.
The notation ``$\sigma \leq \ker i$'' can thus be construed as 
an inequality with respect to the reverse refinement order 
(cf. Definition \ref{def:21}.3, Remark \ref{rem:35}) on the set 
$\cP (2n)$ of all partitions of $\{ 1, \ldots , 2n \}$. }
that $\sigma$ is in $NCP(2n)$ 
and fulfills the compatibility condition 
$(W = \{ p,q \} \in \sigma ) \Rightarrow i(p)= i(q)$.  With this 
notation, (\ref{eqn:59d}) becomes
\begin{equation}  \label{eqn:59e}
\varphi ( a_{i(1)} \cdots a_{i(2n)} ) 
\, = \ \mid \{ \sigma \in NCP(2n)  \mid 
       \sigma \leq \ker i  \} \mid .
\end{equation}  

We now unfix $( i(1), \ldots  , i(2n) )$ and return to Equation
(\ref{eqn:59c}).  We find that
\[
( \varphi \otimes \varphi ) (x^{2n}) = 
\sum_{i(1), \ldots , i(2n) =1}^d \ \mid
\Bigl\{ ( \sigma, \theta ) \in NCP (2n)^2 
\mid  \sigma \leq \ker i \mbox{ and } \theta \leq \ker i 
\Bigr\} \mid 
\]
\begin{equation}  \label{eqn:59f}
= \sum_{\sigma , \theta \in NCP(2n)} \
\mid \{ ( i(1), \ldots  , i(2n) ) \in \{ 1, \ldots ,d \}^{2n} 
\mid \sigma \leq \ker i \mbox{ and } \theta \leq \ker i 
\Bigr\} \mid ,
\end{equation}
where (\ref{eqn:59f}) is obtained via change of order of summation 
in the suitable sum of $0$'s and $1$'s indexed by the aggregated
$\sigma, \theta$ and $( i(1), \ldots  , i(2n) )$.

Let us now momentarily fix $\sigma, \theta \in NCP(2n)$, which we write 
as $A( \pi )$ and respectively $A( \rho )$, with $\pi , \rho \in NC(n)$.
It is immediate that for a tuple 
$( i(1), \ldots  , i(2n) ) \in \{ 1, \ldots ,d \}^{2n}$, the condition
``$\sigma \leq \ker i$ and $\theta \leq \ker i$''
is equivalent to asking that 
$i : \{ 1, \ldots , 2n \} \to \{ 1, \ldots ,d \}$ is constant along 
the orbits of the permutation $M_{\pi , \rho}$.  This clearly implies
\begin{equation}   \label{eqn:59g}
\vline \,  \{ ( i(1), \ldots  , i(2n) ) \mid
\sigma \leq \ker i \mbox{ and }
\theta \leq \ker i \}  \, \vline \ 
= d^{\# ( M_{\pi , \rho} ) }.
\end{equation}

We finally let $\sigma, \theta$ run in $NCP(2n)$ (equivalently, we 
let $\pi , \rho$ run in $NC(n)$) and we replace 
(\ref{eqn:59g}) into (\ref{eqn:59f}), to obtain
\begin{align*}
( \varphi \otimes \varphi ) (x^{2n}) 
& = \sum_{\pi , \rho \in NC(n)} \ d^{\# ( M_{\pi , \rho} ) }       \\
& = \sum_{k=1}^n \mid \{ ( \pi , \rho ) \in NC(n)^2 
\mid \# ( M_{\pi , \rho} ) = k \} \mid  \cdot d^k                  \\
& = \sum_{k=1}^n  \mek_n d^k ,
\end{align*}
as had to be proved.
\end{proof}

$\ $

\begin{remark}    \label{rem:510}
I am grateful to Roland Speicher for bringing to my attention 
the following fact: one can easily adjust the proof of 
Proposition \ref{prop:59} in order to find combinatorial 
interpretations for the moments (with respect to 
$\varphi \otimes \varphi$) of more general elements of the form 
$a_1 \otimes b_1 + \cdots + a_d \otimes b_d \in \cA \otimes \cA$,
where each of $a_1, \ldots , a_d$ and $b_1, \ldots , b_d$ is a 
free family of elements of $( \cA , \varphi )$.  Here are two 
nice examples obtained on these lines.

\vspace{6pt}

(a) Let $( \cA , \varphi )$ and $a_1, \ldots , a_d \in \cA$
be exactly as in Proposition \ref{prop:59}, and let us put
\[
y := a_1^2 \otimes a_1^2 + \cdots + a_d^2 \otimes a_d^2 
\in \cA \otimes \cA .
\]
It is known that $a_i^2$ has free cumulants 
$\kappa_p (a_i^2) = 1$ for all $p \in \bN$.  By using this fact 
and by repeating the method of calculation from the proof of 
Proposition \ref{prop:59}, one finds that 
\begin{equation}  \label{eqn:510a}
( \varphi \otimes \varphi ) (y^n) = 
\sum_{\pi, \rho \in NC(n)} 
d^{ | \pi \widetilde{\vee} \rho | } 
 , \ \ n \in \bN ,
\end{equation}  
where ``$\widetilde{\vee}$'' is the join operation for the 
lattice $\cP (n)$ (cf. Remark \ref{rem:35}).  The occurrence 
of the operation $\widetilde{\vee}$ in connection to partitions 
from $NC(n)$ may seem a bit strange, but matrices of the form 
$\bigl[ \, q^{ | \pi \widetilde{\vee} \rho | } 
\, \bigr]_{\pi, \rho \in NC(n)}$ 
do appear in the research literature -- see e.g. \cite{J1994}.

\vspace{6pt}

(b) With $a_1, \ldots , a_d$ still being exactly as in 
Proposition \ref{prop:59}, let us put
\[
z := a_1 \otimes a_1^2 + \cdots + a_d \otimes a_d^2 
\in \cA \otimes \cA .
\]
It is immediate that one has 
$( \varphi \otimes \varphi ) (z^{2n-1}) = 0$ for all $n \in \bN$.
For the even moments of $z$, the method of calculation from the 
proof of Proposition \ref{prop:59} (and the combined knowledge 
of the free cumulants of $a_i$ and $a_i^2$) leads to the formula
\begin{equation}  \label{eqn:510b}
( \varphi \otimes \varphi ) (z^{2n}) = 
\sum_{ \begin{array}{c}
{\scriptstyle \sigma \in NC(2n)}   \\
{\scriptstyle \theta \in NCP(2n)}  
\end{array} } \
d^{ | \sigma \widetilde{\vee} \theta | } 
 , \ \ n \in \bN ,
\end{equation}  
a version of (\ref{eqn:510a}) which now mixes together 
non-crossing pairings with general non-crossing partitions 
from $NC(2n)$.
\end{remark} 

$\ $

$\ $

\noindent
{\bf\Large Acknowledgement}

\vspace{4pt}

\noindent
This paper was written while I was on a long-term visit to MAP5,
the Applied Mathematics Department of Universit\'e Paris 5.
I acknowledge the hospitality of MAP5, and many interesting 
discussions with Florent Benaych-Georges,
Thierry Cabanal-Duvillard and Camille Male
around the topic of the paper.

$\ $

$\ $

$\ $

$\ $

Alexandru Nica

Department of Pure Mathematics, University of Waterloo,

Waterloo, Ontario N2L 3G1, Canada.

Email: anica@uwaterloo.ca

\end{document}